\documentclass[a4paper,draft]{amsproc}
\usepackage{amssymb}
\usepackage{amscd} %% Package for commutative diagrams
\usepackage{snapshot}
\usepackage[dvips]{graphicx}
\usepackage{color}
\usepackage{subfigure}
\usepackage{rotating}
\usepackage{amsmath}
\usepackage{amsthm}
\usepackage{amsfonts}
\usepackage{stmaryrd}
\theoremstyle{plain}
 \newtheorem{theorem}{Theorem}[section]
 
 \newtheorem{lemma}{Lemma}[section]
 \newtheorem{corollary}{Corollary}[section]

 \newtheorem{definition}{Definition}[section]
\theoremstyle{remark}
 
 \numberwithin{equation}{section}

\renewcommand{\leq}{\leqslant}
\renewcommand{\geq}{\geqslant}

\title[]{Law of total probability and Bayes' theorem in Riesz spaces}



\author[Hong]{ Liang Hong}

\address{
Department of Mathematics \\ % \hfill (Received 00 00 2010)\\
Robert Morris University   \\ %\hfill (Revised  00 00 2010)\\
Moon, PA 15108, USA}
\email{hong@rmu.edu}

\begin{document}

\vspace{18mm}
\setcounter{page}{1}
\thispagestyle{empty}

\begin{center}
%Submitted to \emph{Indagationes Mathematicae}
\end{center}

\begin{abstract}
This note generalizes the notion of conditional probability to Riesz spaces using the order-theoretic approach. With the aid of this concept, we establish the law of total probability and Bayes' theorem in Riesz spaces; we also prove an inclusion-exclusion formula in Riesz spaces. Several examples are provided to show that the law of total probability, Bayes' theorem and inclusion-exclusion formula in probability theory are special cases of our results.
\end{abstract}

\maketitle

\section{Introduction}
The study of stochastic processes in an abstract space can date back to as early as \cite{DeMARR}. Works along this line include \cite{GT}, \cite{Ghoussoub}, \cite{Hong}, \cite{Stoica1} and \cite{Stoica2}. About a decade ago, \cite{KLW1} initiated the study of stochastic processes in a measure-free setting. They carefully investigated the fundamental properties of conditional expectations on a probability space and abstract them to Riesz spaces; this allows one to generalize many concepts from probability theory to Riesz spaces. As a result, an order-theoretic approach to stochastic process can be formed to generalize the classical theory of stochastic processes. Since then this order-theoretic approach to stochastic processes has been developing fast. For order-theoretic investigation of conditional expectations and related concepts, we refer to \cite{GP}, \cite{Hong}, \cite{KLW4},  \cite{Watson1}; for discrete-time processes in Riesz spaces, we refer to \cite{KL1}, \cite{KLW1}, \cite{KLW2}, \cite{KLW3}, \cite{KLW5}, \cite{KLW6}, \cite{KVW1},  \cite{Troitsky}; for continuous-time processes in Riesz spaces, we refer to \cite{Grobler1}, \cite{Grobler2}, \cite{Grobler3}, \cite{Grobler4}, \cite{Grobler5}, \cite{GLM1}, \cite{VW1}, \cite{VW2}; for stochastic integrals in Riesz spaces, we refer to \cite{GL1}, \cite{GL2}, \cite{LW1}.

\cite{KLW6} extended the notation of independence to Riesz spaces; \cite{Grobler4} gave a slightly more general definition of independence in Riesz spaces. In this note, we follow \cite{KLW6} to extend the concept of conditional probability to Riesz spaces and then establish the law of total probability and Bayes' theorem in Riesz spaces; we also prove an inclusion-exclusion formula in Riesz spaces. Several examples are given to show that the concept of conditional probability, law of total probability, Bayes' theorem and inclusion-exclusion formula in probability theory are special cases of our results.

The remainder of this note is organized as follows. Section 2 review some basic concepts and results of Riesz spaces and conditional expectations in Riesz spaces; for further details, we refer readers to  \cite{AB2}, \cite{Fuchs1}, \cite{KLW4}, \cite{LZ}, \cite{Watson1} and \cite{Zaanen}. Section 3 generalizes the concept of conditional probability to Riesz spaces; we show that this notion leads naturally to the notion of independence in Riesz spaces. Section 4 establishes the law of total probability, Bayes' theorem and the inclusion-exclusion formula in Riesz spaces.

\section{Preliminaries of Riesz space theory}
A partially ordered set $(X, \leq)$ is called a \emph{lattice} if the infimum and supremum of any pair of elements in $X$ exist. A real vector space $X$ equipped with a partial order $\leq$ is called an \emph{ordered vector space} if its vector space structure is compatible with the order structure in a manner such that
\begin{enumerate}
  \item [(a)]if $x\leq y$, then $x+z\leq y+z$ for any $z\in X$;
  \item [(b)]if $x\leq y$, then $\alpha x\leq \alpha y$ for all $\alpha\geq 0$.
\end{enumerate}
An ordered vector space is called a \emph{Riesz space} (or a \emph{vector lattice}) if it is also a lattice at the same time. For any pair $x, y$ in a Riesz space, $x\vee y$ denotes and supremum of $\{x, y\}$, $x\wedge y$ denotes the infimum of $\{x, y\}$, and $|x|$ denotes $x\vee (-x)$. A Riesz space $X$ is said to be \emph{Dedekind complete} if every nonempty subset of $X$ that is bounded from above has a supremum. A vector subspace of a Riesz space is said to be a \emph{Riesz subspace} if it is closed under the lattice operation $\vee$. A subset $Y$ of a Riesz space $X$ is said to be \emph{solid} if $|x| \leq |y|$ and $y\in Y$ imply that $x\in Y$. A solid vector subspace of a Riesz space is called an \emph{ideal}. A net $(x_{\alpha})_{\alpha\in A}$ is said to be \emph{decreasing} if $\alpha\geq \beta$ implies $x_{\alpha}\leq x_{\beta}$. The notation $x_{\alpha}\downarrow x$ means $(x_{\alpha})_{\alpha\in A}$ is a decreasing net and the infimum of the set $\{x_{\alpha}\mid \alpha\in A\}$ is $x$. A net $(x_{\alpha})_{\alpha\in A}$ in a Riesz space $X$ is said to be \emph{order-convergent} to an element $x\in X$, often written as $x_{\alpha}\xrightarrow{o} x$,  if there exists another net $(y_{\alpha})_{\alpha\in A}$ in $X$ such that $|x_{\alpha}-x|\leq y_{\alpha}\downarrow 0$. A subset $A$ of a Riesz space $X$ is said to be \emph{order-closed} if a net $\{x_{\alpha}\}$ in $A$ order-converges to $x_{0}\in X$ implies that $x_0\in A$. An order-closed ideal is called a \emph{band}. Two elements $x$ and $y$ in a Riesz space $X$ is said to be \emph{disjoint} if $|x|\wedge |y|=0$ and is denoted by $x\bot y$. The \emph{disjoint complement} of a subset $A$ of a Riesz space $X$ is defined as $A^d=\{x\in X\mid x\wedge y\text{ for all $y\in A$}\}$. A band $B$ in a Riesz space is said to be a \emph{projection band} if $X$ equals the direct sum of $B$ and $B^d$, that is, $X=B\oplus B^d$. Every band in a Dedekind complete Riesz space is a projection band. If $B$ is a projection band of a Riesz space $X$ and $x\in X$, then $x=x_1+x_2$, where $x_1\in B$ and $x_2\in B^d$. In this case, the map $P_B: X\rightarrow X$ defined by $P_B(x)=x_1$ is called the \emph{band projection} associated with $B$.

%Let $A$ be a nonempty subset of a Riesz space $X$. The ideal generated by $A$ is the smallest ideal containing $A$ and can be described as
%$X_A=\{x\in X\mid |x|\leq \alpha \sum_{i=1}^n|x_i|, \text{where} \ \alpha>0, x_1, ..., x_n\in A\}$. Likewise, the band generated by $A$ is the smallest band containing $A$ and is often %denoted as $B_A$.

An additive group $G,$ equipped with a partial order $\leq$ is called a \emph{partially ordered group} if its group operation is compatible with the order structure in a manner such that
\begin{enumerate}
  \item [(a)]if $x\leq y$, then $x+z\leq y+z$ for any $z\in G$;
  \item [(b)]if $x\leq y$, then $z+x\leq z+y$ for any $z\in G$.
\end{enumerate}
A partially ordered group is said to be a \emph{lattice-ordered group} if it is also a lattice at the same time. Lattice-ordered rings and lattice-ordered fields are defined in a similar manner. Specifically, a ring $R$ equipped with a partial order $\leq$ is called a \emph{partially ordered ring} if its ring operations are compatible with the order structure in a manner such that
\begin{enumerate}
  \item [(a)]if $x\leq y$, then $x+z\leq y+z$ for any $z\in R$;
  \item [(b)]if $x\leq y$ and $z>0$, then $x z\leq y z$ and $z x\leq z y$.
\end{enumerate}
A partially ordered ring is said to be a \emph{lattice-ordered ring} if its additive group is a lattice-ordered group. Likewise, a field $F$ equipped with a partial order $\leq$ is called a \emph{partially ordered field} if its field operations are compatible with the order structure in a manner such that
\begin{enumerate}
  \item [(a)]if $x\leq y$, then $x+z\leq y+z$ for any $z\in F$;
  \item [(b)]if $x> y$ and $z>0$, then $x z>y  z$;
  \item [(c)]$1>0$, where $1$ is the identity element of $F$.
\end{enumerate}
A partially ordered field is said to be a \emph{lattice-ordered field} if its additive group is a lattice-ordered group.

A linear operator $T$ on a vector space $X$ is said to be a \emph{projection} if $T^2=T$. A linear operator $T$ between two Riesz spaces $X$ and $Y$ is said to be \emph{positive} if $x\in X$ and $x\geq 0$ implies $T(x)\geq 0$; $T$ is said to be \emph{strictly positive} if $x\in X$ and $x>0$ implies $T(x)>0$; $T$ is said to be \emph{order-continuous} if $x_{\alpha}\xrightarrow{o} 0$ in $X$ implies $T(x_{\alpha})\xrightarrow{o} 0$ in $Y$. A linear operator $T$ on a vector space $V$ is called an \emph{averaging operator} if $T(y T(x))=T(y)T(x)$ for any pair $x, y \in V$ such that $y T(x)\in V$.

Following \cite{KLW1} and \cite{KLW4}, we say a positive order-continuous projection on a Riesz space $X$ with a weak order unit $e$ is a \emph{conditional expectation} if the range $R(T)$ of $T$ is a Dedekind complete Riesz subspace of $X$ and $T(e)$ is a weak order unit of $X$ for every weak order unit $e$ of $X$. If $T$ is a conditional expectation on a Riesz space $X$, then there is a weak order unit $e$ such that $T(e)=e$.

\section{Conditional probability in Riesz spaces}
To extend the notion of conditional probability in probability theory to Riesz spaces, we first recall the following definition.
\begin{definition}\label{definition2.0}
Let $(\Omega, \mathcal{F}, P)$ be a probability space. Suppose $A$ and $B$ be two events such that $P(B)>0$. Then the conditional probability $P(A\mid B)$ of $A$ given $B$ is defined as
\begin{equation*}
P(A\mid B)=\frac{P(A\cap B)}{P(B)}.
\end{equation*}
\end{definition}
Intuitively, $P(A\mid B)$ measures the relative size of the projection of $A$ onto $B$. This motivates the following definition.

\begin{definition}\label{definition2.1}
Let $E$ be a Dedekind complete lattice-ordered field with a conditional expectation $T$ and a weak unit $e$ such that $T(e)=e$. Suppose $B_1$ and $B_2$ are two projection bands, $P_{B_1}$ and $P_{B_2}$ are the corresponding band projections of $B_1$ and $B_2$, respectively, and $B_1\cap B_2\neq \phi$. Then the conditional probability $P_{B_1\mid B_2}$ of $B_1$ given $B_2$ with respect to $T$ is defined as
\begin{equation}\label{2.1}
P(B_1\mid B_2)(f)=[TP_{B_2}P_{B_1}(f)][TP_{B_2}(f)]^{-1}, \quad \forall f\in E.
\end{equation}
provided $TP_{B_2}(f)\neq 0$.
\end{definition}
\noindent \textbf{Remark.} We require $E$ to be a lattice-ordered field because we need the multiplication operation on $E$ to be invertible. It is evident that if $P_{B_1}=P_{B_2}$ then $P_{B_1\mid B_2}=I$, where $I$ is the identity operator on $E$. Since two band projections on a Riesz space are commutative, Equation (\ref{2.1}) may be replaced by
% See Theorem 1.45 of \cite{AB2}
\begin{equation}\label{2.2}
P(B_1\mid B_2)(f)=\left[TP_{B_1}P_{B_2} (f)\right][TP_{B_2}(f)]^{-1}, \quad \forall f\in  E,\nonumber
\end{equation}
provided $TP_{B_2}(f)\neq 0$.
\bigskip

The following example shows that the classical definition of conditional probability is a special case of Definition 2.1.\\

\noindent \textbf{Example 2.1.} Let $E$ be the family of all random variables with finite mean on a probability space $(\Omega, \mathcal{F}, P)$, that is, $E=L^1(\Omega, \mathcal{F}, P)$. Then $e=1_{\Omega}$, where $1_{A}$ denotes the indicator function of an event $A$. Let $T$ be the expectation operator. For two events $A, B\in\mathcal{F}$ with $B\neq \phi$, the bands generated by $1_A$ and $1_B$ are $B_1=L^1(A, \mathcal{F}\cap A, P)$ and $B_2=L^1(B, \mathcal{F}\cap B, P)$, respectively. Therefore, the corresponding band projections $P_{B_1}$ and $P_{B_2}$ are defined as $P_{B_1}(f)=f1_A$ and $P_{B_2}(f)=f1_B$. We have
\begin{equation*}
TP_{B_2}(e)=E[1_{\Omega}1_B]=P(B),
\end{equation*}
and
\begin{equation*}
TP_{B_2}P_{B_1}(e)=E[1_{\Omega}1_B1_A]=P(A\cap B).
\end{equation*}
It follows that $P(B_1\mid B_2)=P(A\mid B)$ which is consistent with the classical definition of conditional probability of $A$ given $B$. \\

Recall that in probability theory two events $A$ and $B$ are defined to be \emph{independent} if $P(A\mid B)$ does not depend on $B$ and $P(B\mid A)$ does not depend on $A$ , that is $P(A\mid B)=P(A)$ and $P(B\mid A)=P(B)$. To accommodate the case $P(B)=0$, we usually define two events $A$ and $B$ to be \emph{independent} if $P(A\cap B)=P(A)P(B)$. Thus, Definition \ref{definition2.1} leads naturally to the following concept of independence in Riesz spaces.

\begin{definition}\label{definition2.2}
Let $E$ be a Dedekind complete lattice-ordered ring with a conditional expectation $T$ and a weak unit $e$ such that $T(e)=e$ and the multiplication operation coincide with the lattice operation $\wedge$. Suppose $B_1$ and $B_2$ are two projection bands of $E$ and $P_{B_1}$ and $P_{B_2}$ are the associated band projections of $B_1$ and $B_2$, respectively. We say $B_1$ and $B_2$ are independent with respect to $T$ if $TP_{B_2}P_{B_1}(e)=TP_{B_2}(e)TP_{B_1}(e)$.
\end{definition}
\noindent \textbf{Remark.} In Definition \ref{definition2.2}, we assume that $E$ is a Dedekind complete lattice-ordered ring and its multiplication is the lattice operation $\wedge$. Definition 4.1 in \cite{KLW6} defines two projection bands $P_{B_1}$ and $P_{B_2}$ to be independent with respect to a conditional expectation $T$ in a Dedekind complete Riesz space if
\begin{equation*}
TP_{B_1}TP_{B_2}(e)=TP_{B_1}P_{B_2}(e)=TP_{B_2}TP_{B_1}(e).
\end{equation*}
As \cite{KVW1} pointed out, the space of $e$ bounded elements in $E$ will admit a natural f-algebra structure if we set $P(e)Q(e)=PQ(e)$ for two band projections $P$ and $Q$. In view of this, our definition and Definition 4.1 in \cite{KLW6} are consistent.

%It is interesting to note that their definition is slightly different from ours. If the two definitions were the same then we would have $TP_1TP_2(e)=TP_1(e)TP_2(e)$. By Theorem 2.2 of \cite{KLW6} and the fact $T(e)=e$, this would imply $P_2P_1(e)=P_2(e)P_1(e)$; hence $P_2 (e)\wedge P_1(e)=P_2(e)P_1(e)$ which would not hold in general. The two definitions are the same if the ring operation coincides with the lattice operation $\wedge$. Also, the two definitions are the same in the setup of probability theory as the next example illustrates.\\
% cf problems 8 and 9 on p45 of AB(1985).

%Alternatively, Since $T$ is an averaging operator (Theorem 4.3 of \cite{KLW4}), this would imply $TP_1TP_2e=T(P_2eTP_1e)$. By Theorem 2.2 of \cite{KLW6} and positivity of $T$, we would have $P_2P_1e=P_2eP_1e$.
%$TTP_2P_1e=T(P_2eTP_1e)=>TP_2P_1e=T(P_2eTP_1e)=>P_2P_1e=P_2eTP_1e=P_2eP_1Te=P_2eP_1e$ as $Te=e$.

\section{Law of total probability and Bayes' theorem in Riesz spaces}
In probability theory, the law of total probability and Bayes' theorem are two fundamental theorems involving conditional probability. In this section,
we will extend these two theorems to Riesz spaces. We will also give examples to show that the classical law of total probability and Bayes' theorem are special cases
of law of total probability and Bayes' theorem in Riesz spaces. To motivate our study, we first recall the \emph{law of total probability} in probability theory:
\begin{theorem}\label{theorem3.1}
Suppose $B_1, ..., B_n$ are mutually disjoint events on a probability space $(\Omega, \mathcal{F}, P)$ such that $\Omega=\cup_{i=1}B_i$ and $P(B_i)>0$ for each $1\leq i\leq n$. Then for any event $A$
\begin{equation*}
P(A)=\sum_{i=1}^nP(A\mid B_i)P(B_i).
\end{equation*}
\end{theorem}
Since $P(A\mid B_i)P(B_i)=P(A\cap B_i)$, the intuition of law of total probability can be described as follows: If an event $A$ is the  union of $n$ mutually disjoint sub-events $A\cap B_i$, then the projection onto $A$ is equivalent to the sum of all projections onto $A\cap B_i$'s, that is,
\begin{equation*}
P(A)=\sum_{i=1}^nP(A\cap B_i).
\end{equation*}
This intuition extends naturally to the abstract case of Riesz spaces as the next theorem shows.
To this end, we first state and prove an inclusion-exclusion formula for projection bands; this result extends the classical inclusion-exclusion theorem
for events to Riesz spaces. \\

\begin{lemma}[Inclusion-exclusion formula for projection bands]\label{lemma3.1}
Let $B_1, ..., B_n$ be $n$ projection bands in a Riesz space. Then $B_1+...+B_n$ is a projection band and
\begin{eqnarray}\label{2.4}
P_{B_1+...+B_n} &=& \sum_{i=1}^nP_{B_i}-\sum_{1\leq i_1<i_2\leq n} P_{B_{i_1}\cap B_{i_2}} \nonumber\\
                & & +...+(-1)^{k-1}\sum_{1\leq i_1<...<i_k\leq n} P_{B_{i_1}\cap...\cap B_{i_k}}\nonumber \\
                & & +...+(-1)^{n-1} P_{B_1\cap...\cap B_n}.
\end{eqnarray}
\end{lemma}

\begin{proof}
The fact that $B_1+...+B_n$ is a projection band is trivial. We show Equation (\ref{2.4}) by induction on $n$. The case $n=2$ follows from Theorem 1.45 of \cite{AB2}. Suppose Equation (\ref{2.4}) holds for $n-1$. Then the induction hypothesis implies
\begin{eqnarray*}
P_{B_1+...+B_n} &=& P_{B_1+...+B_{n-1}}+P_{B_n}-P_{B_1+...+B_{n-1}}P_{B_n} \\
                &=& \sum_{i=1}^{n-1}P_{B_i}-\sum_{1\leq i_1<i_2\leq n} P_{B_{i_1}\cap B_{i_2}}+...\\
                & & +(-1)^{k-1}\sum_{1\leq i_1<...<i_k\leq n} P_{B_{i_1}\cap...\cap B_{i_k}}+...+(-1)^{n-2}P_{B_1\cap...\cap B_{n-1}}\\
                & & +P_{B_n}-P_{B_1+...+B_{n-1}}P_{B_n}\\
                &=& \sum_{i=1}^{n-1}P_{B_i}-\sum_{1\leq i_1<i_2\leq n} P_{B_{i_1}\cap B_{i_2}}+...\\
                & & +(-1)^{k-1}\sum_{1\leq i_1<...<i_{k}\leq n} P_{B_{i_1}\cap...\cap B_{i_k}}+...+(-1)^{n-2}P_{B_1\cap...\cap B_{n-1}}\\
                & & +P_{B_n}-\bigg \lbrack \sum_{i=1}^{n-1}P_{B_i}-\sum_{1\leq i_1<i_2\leq n} P_{B_{i_1}\cap B_{i_2}}+...\\
                & & +(-1)^{k-2}\sum_{1\leq i_1<...<i_{k-1}\leq n} P_{B_{i_1}\cap...\cap B_{i_{k-1}}}+...+(-1)^{n-2}P_{B_1\cap...\cap B_{n-1}} \bigg \rbrack P_{B_n}\\
                &=& \sum_{i=1}^{n-1}P_{B_i}+P_{B_n} \\
                & & -(\sum_{i=1}^{n-1} P_{B_i})P_{B_n}-\sum_{1\leq i_1<i_2\leq n}} P_{B_{i_1}\cap B_{i_2}\\
                & & +...+\\
                & & (-1)^{k-1}\sum_{1\leq i_1<...<i_k\leq n} P_{B_{i_1}\cap...\cap B_{i_k}}+(-1)^{k-1}\sum_{1\leq i_1<...<i_{k-1}\leq n} P_{B_{i_1}\cap...\cap B_{i_{k-1}}}P_{B_n}\\
                & & +...+\\
                & & +(-1)^{n-1}P_{B_1\cap ...\cap B_{n-1}}P_{B_n}\\
                &=& \sum_{i=1}^nP_{B_i}-\sum_{1\leq i_1<i_2\leq n} P_{B_{i_1}\cap B_{i_2}}\\ \nonumber
                & & +...+(-1)^{k-1}\sum_{1\leq i_1<...<i_k\leq n} P_{B_{i_1}\cap...\cap B_{i_k}}\\ \nonumber
                & & +...+(-1)^{n-1} P_{B_1\cap...\cap B_n}.
\end{eqnarray*}
By mathematical induction, (\ref{2.4}) holds for all positive integer $n$.
\end{proof}

\begin{theorem}[Law of total probability in Riesz spaces]\label{theorem3.2}
Let $E$ be a Dedekind complete lattice-ordered field with a conditional expectation $T$ and a weak unit $e$ such that $T(e)=e$. Suppose $B_1, ..., B_n$ are $n$ mutually disjoint nonempty projection bands such that $E=B_1\oplus B_2 \oplus...\oplus B_n$and $P_{B_i}$ is the band projection onto $B_i$. Then for any projection band $D$ and its associated band projection $P_D$
\begin{equation}\label{2.3}
TP_D(f)=\sum_{i=1}^n (P_{D\mid B_i}(f))(TP_{B_i}(f)),\quad \forall f\in E.
\end{equation}
provided $TP_{B_i}(f)\neq 0$ for all $1\leq i\leq n$.
\end{theorem}

\begin{proof}
Definition \ref{definition2.1} shows that the right-hand side of Equation (\ref{2.3}) equals
\begin{eqnarray*}
\sum_{i=1}^n T(P_D P_{B_i}) &=& T \left[\sum_{i=1}^n (P_{D\cap B_i})\right]\nonumber
\end{eqnarray*}
Since $B_i$'s are mutually disjoint, $P_{D\cap B_i}\cap P_{D\cap B_j}=0$ for $i\neq j$. It follows from Lemma \ref{lemma3.1} that
\begin{eqnarray*}
T \left[\sum_{i=1}^n (P_{D\cap B_i})\right] & = & T \left[ (P_{D\cap (\sum_{i=1}^n B_i)} \right] \\
& =& TP_D.
\end{eqnarray*}
Therefore, the theorem is established.
\end{proof}

Theorem \ref{theorem3.1} and Definition \ref{definition2.1} immediately imply the following Bayes' Theorem in Riesz spaces.

\begin{corollary}[Bayes' Theorem in Riezs spaces]\label{corollary3.1}
Let $E$ be a Dedekind complete lattice-ordered field with a conditional expectation $T$ and a weak unit $e$ such that $Te=e$. Suppose $B_1, ..., B_n$ are $n$ mutually disjoint nonempty projection bands such that $E=B_1\oplus B_2 \oplus...\oplus B_n$and $P_{B_i}$ is the band projection onto $B_i (1\leq i\leq n)$. Then for any projection band $D$ and any $1\leq j\leq n$
\begin{equation}\label{2.5}
P_{B_j\mid D}(f)=\left[(P_{D\mid B_j}(f))(TP_{B_j}(f))\right] \left[\sum_{i=1}^n (P_{D\mid B_i}(f))(TP_{B_i}(f))\right]^{-1},
\quad \forall f\in E, \nonumber
\end{equation}
provided $TP_{B_i}(f)\neq 0$ for all $1\leq i\leq n$.
\end{corollary}

\bigskip
\noindent \textbf{Example 3.1.} Take $E=L^1(\Omega, \mathcal{F}, P)$. Then $e=1_{\Omega}$. Suppose $T$ is the expectation operator and $P_{B_i}$'s are defined as $P_{B_i}(f)=1_{B_i}f$ for $f\in E$, where $B_1, B_2, ..., B_n$ are $n$ disjoint nonempty events such that $\Omega=\cup_{i=1}^n B_i$. Also, for any event $D$ define $P_D(f)=1_Df$ for $f\in E$. Then $TP_D(e)=P(D), TP_{B_i}(e)=P(B_i)$ and $P_{D\mid B_i}(e)=P(D\mid B_i)$; thus, Theorem \ref{theorem3.2} specializes to Theorem \ref{theorem3.1}. Moreover, we have $P_{B_1+...+B_n}(e)=P(\cup_{i=1}^n B_i), P_{B_i} (e)=P(B_i), P_{B_{i_1} \cap B_{i_2}} (e)=P(B_{i_1}\cap B_{i_2}), ...$; hence Lemma \ref{lemma3.1} specializes to the classical inclusion-exclusion formula. Finally, in this case Corollary \ref{corollary3.1} yields
\begin{equation*}
P(B_j\mid D)=\frac{P(D\mid B_j) P(B_j)}{\sum_{i=1}^n P(D\mid B_i)P(B_i)};
\end{equation*}
therefore, Corollary \ref{corollary3.1} specializes to the classical Bayes' theorem.

\section*{Acknowledgements} Thanks are due to Professor Coenraad C.A. Labuschagne for providing the author with some references.

\bibliographystyle{amsplain}

\begin{thebibliography}{n} %% n is number of items, or the largest label

\bibitem{AB2} Aliprantis, C.D. and Burkinshaw, O. (1985). \emph{Positive Operators}, Springer, Berlin, Heidelberg, New York.

\bibitem{DeMARR} DeMarr, R. (1965). A martingale convergence theorem in vector lattices, \emph{Canadian Journal of Mathematics}, 18, 424-432.


\bibitem{Fuchs1} Fuchs, L. (1963). \emph{Partially Ordered Algebraic Systems}, Pergamon Press, Oxford, New York.

\bibitem{GT} Gessesse, H.E. and Troitsky, V.G. (2005). Martingales in Banach Lattices, II, \emph{Positivity}, 15, 49-55.


\bibitem{Ghoussoub} Ghoussoub, N. (1982). Riesz spaces valued measures and processes, \emph{Bulletin De La S.M.F.}, 110, 233-257.

\bibitem{Grobler1} Grobler, J.J. (2010). Continuous stochastic processes in Riesz spaces: the Doob-Meyer decomposition, \emph{Positivity}, 14, 731-751.

\bibitem{Grobler2} Grobler, J.J. (2011). Doob's optional sampling theorem in Riesz spaces, \emph{Positivity}, 15, 617-637.

\bibitem{Grobler3} Grobler, J.J. (2014). Jensen's and martingale inequalities in Riesz spaces, \emph{Indagationes Mathematicae}, 24, 275-295.

\bibitem{Grobler4} Grobler, J.J. (2014). The Kolmogorov-$\check{C}$entsov theorem and Brownian motion in vector lattices, \emph{Journal of Mathematical Analysis and Applications}, 410, 891-901.

\bibitem{Grobler5} Grobler, J.J. (2014). Corrigendum to ``The Kolmogorov-$\check{C}$entsov theorem and Brownian motion in vector lattices [J. Math. Anal. Appl. 410 (2014) 891-901]'', \emph{Journal of Mathematical Analysis and Applications}, 420, 878.

\bibitem{GL1} Grobler, J.J. and Labuschagne, C.A. (2015). The It$\hat{o}$ integral for Brownian motion in vector lattice: Part 1, Grobler, J.J., \emph{Journal of Mathematical Analysis and Applications}, 423, 797-819.

\bibitem{GL2} Grobler, J.J. and Labuschagne, C.A. (2015). The It$\hat{o}$ integral for Brownian motion in vector lattice: Part 2, Grobler, J.J., \emph{Journal of Mathematical Analysis and Applications}, 423, 820-833.


\bibitem{GLM1} Grobler, J.J., Labuschagne, C.A. and Marraffa, V. (2014). Quadratic variation of martingales in Riesz spaces,  \emph{Journal of Mathematical Analysis and Applications}, 410, 418-426.

\bibitem{GP} Grobler, J.J. and Pagter, B. (2002). Operators representable as multiplication-conditional expectation operators, \emph{Journal of Operator Theory}, 141, 55-77.


\bibitem{Hong} Hong, L. (2015). And$\hat{o}$-Douglas type characterization of optional projections and predictable projections, \emph{Indagationes Mathematicae}, http://dx.doi.org/10.1016/j.indag.2015.01.001.

\bibitem{KL1} Korostenski, M. and Labuschagne, C.A. (2008). A note on regular martingales in Riesz spaces, \emph{Quaestiones Mathematicae}, 31, 219-224.

\bibitem{KLW1}Kuo, W.C., Labuschagne, C.A. and Watson, B.A. (2004). Discrete-time stochastic processes on Riesz spaces, \emph{Indagationes Mathematicae}, 15 (3), 435-451.

\bibitem{KLW2} Kuo, W.C., Labuschagne, C.A. and Watson, B.A. (2004). Riesz space and fuzzy upcrossing theorems, in: Soft Methodology and Random Information System, \emph{Advances in Soft Computing}, 101-108, Springer-Verlag, Berlin, Heidelberg, New York.

\bibitem{KLW3} Kuo, W.C., Labuschagne, C.A. and Watson, B.A. (2005). Zero-one laws for Riesz space and fuzzy random variables, \emph{Fuzzy logic, soft computing and computational intelligence}, 393-397, Springer-Verlag and Tsinghua University Press, Beijing.

\bibitem{KLW4} Kuo, W.C., Labuschagne, C.A. and Watson, B.A. (2005). Conditional expectations on Riesz spaces, \emph{Journal of Mathematical Analysis and Applications}, 303, 509-521.

\bibitem{KLW5} Kuo, W.C., Labuschagne, C.A. and Watson, B.A. (2006). Convergence of Riesz space martingales, \emph{Indagationes Mathematicae} (New Series). 17 (2), 271-283.

\bibitem{KLW6} Kuo, W.C., Labuschagne, C.A. and Watson, B.A. (2007). Ergodic theory and the strong law of large numbers on Riesz spaces, \emph{Journal of Mathematical Analysis and Applications}, 325, 422-437.

\bibitem{KVW1} Kuo, W.C., Vardy, J.J. and Watson, B.A. (2013). Mixingales on Riesz spaces, \emph{Journal of Mathematical Analysis and Applications}, 402, 731-738.


\bibitem{LW1} Labuschagne, C.A. and Watson, B.A. (2010). Discrete stochastic integration on Riesz spaces, \emph{Positivity}, 14, 859-879.


\bibitem{LZ} Luxemburg, W.A.J. and Zaanen, A.C. (1971). \emph{Riesz Spaces, I}, North-Holland, Amsterdam.

\bibitem{Stoica1} Stoica, GH. (1990). Martingales in vector lattices, \emph{Bull. MATH. de la Soc. Sci. Math. de Roumanie}, 34, 4, 357-362.

\bibitem{Stoica2} Stoica, GH. (1991). Martingales in vector lattices II, \emph{Bull. MATH. de la Soc. Sci. Math. de Roumanie}, 35, 1-2, 155-158.


\bibitem{Troitsky} Troitsky, V.G. (2005). Martingales in Banach Lattices, \emph{Positivity}, 9, 437-456.

\bibitem{VW1} Vardy, J.J. and Watson, B.A. (2012). Markov processes on Riesz spaces, \emph{Positivity}, 16, 373-391.

\bibitem{VW2} Vardy, J.J. and Watson, B.A. (2012). Erratum to: Markov processes on Riesz spaces, \emph{Positivity}, 16, 393.

\bibitem{Watson1} Watson, B.A. (2009). An And$\hat{o}$-Douglas type theorem in Riesz spaces with a conditional expecdtation, \emph{Positivity}, 13, 543-558.

\bibitem{Zaanen} Zaanen, A.C.(1997). \emph{Introduction to Operator Theory in Riesz Spaces}, Springer, Berlin, Heidelberg, New York.

\end{thebibliography}

\end{document}